\documentclass[11pt]{article}

\usepackage{fullpage}
\usepackage{natbib}
\usepackage{algorithm}
\usepackage[noend]{algorithmic}
\usepackage{amsmath,amsthm,amsfonts}
\usepackage{amsmath}
\usepackage{hyperref}
\usepackage{color}
\usepackage{mathrsfs}
\usepackage{enumitem}
\usepackage{bm}
\usepackage{multirow}
\usepackage{booktabs}
\usepackage{makecell}
\usepackage{graphicx}
\usepackage{subfigure}
\usepackage{caption}
\usepackage{thmtools}
\usepackage{thm-restate}

\newcommand{\defeq}{\mathrel{\mathop:}=}

\newcommand{\vect}[1]{\ensuremath{\mathbf{#1}}}
\newcommand{\mat}[1]{\ensuremath{\mathbf{#1}}}

\newcommand{\norm}[1]{\|{#1} \|}

\newcommand{\tr}{\mathrm{tr}}

\newcommand{\trans}{^{\top}}

\newcommand{\N}{\mathbb{N}}
\newcommand{\R}{\mathbb{R}}
\renewcommand{\S}{\mathbb{S}}

\newcommand{\E}{\mathbb{E}}
\renewcommand{\Pr}{\mathbb{P}}

\newcommand{\F}{\mathcal{F}}

\newcommand{\A}{\mat{A}}
\newcommand{\B}{\mat{B}}

\newcommand{\I}{\mat{I}}

\newcommand{\X}{\mat{X}}
\newcommand{\Y}{\mat{Y}}

\newcommand{\e}{\vect{e}}

\renewcommand{\v}{\vect{v}}

\newtheorem{theorem}{Theorem}
\newtheorem{lemma}[theorem]{Lemma}
\newtheorem{corollary}[theorem]{Corollary}

\theoremstyle{definition}
\newtheorem{definition}{Definition}
\newtheorem{condition}[definition]{Condition}

\begin{document}

\title{A Short Note on Concentration Inequalities for Random Vectors with SubGaussian Norm}

\author{Chi Jin \\ University of California, Berkeley \\ \texttt{chijin@cs.berkeley.edu}
	\and 
	Praneeth Netrapalli \\ Microsoft Research, India \\ \texttt{praneeth@microsoft.com}
	\and
	Rong Ge \\ Duke University \\
	\texttt{rongge@cs.duke.edu}
	\and
	Sham M. Kakade \\ University of Washington, Seattle \\
	\texttt{sham@cs.washington.edu}
	\and
	Michael I. Jordan \\ University of California, Berkeley \\ \texttt{jordan@cs.berkeley.edu}}

\maketitle

\newcommand{\cnote}{\textcolor[rgb]{1,0,0}{C: }\textcolor[rgb]{1,0,1}}

\newcommand{\la}{\langle}
\newcommand{\ra}{\rangle}
\newcommand{\In}{\mathbb{I}}

\newcommand{\nSG}{\text{nSG}}
\newcommand{\subG}{\text{sub-Gaussian}}
\newcommand{\subE}{\text{sub-Exp}}

\begin{abstract}
In this note, we derive concentration inequalities for random vectors with subGaussian norm (a generalization of both subGaussian random vectors and norm bounded random vectors), which are tight up to logarithmic factors.
\end{abstract}

\section{Introduction}
Concentration (large deviation) inequalities are one of the most important subjects of study in probability theory. A class of distributions for which sharp concentration inequalities have been developed is the class of subGaussian distributions.
\begin{definition}
A random variable $X \in R$ is subGaussian, if there exists $\sigma \in R$ so that:
\begin{equation*}
\E e^{\theta (X - \E X)} \le e^{\frac{\theta^2\sigma^2}{2}}, \quad \forall \theta \in \R.
\end{equation*}
\end{definition}
\begin{definition}
A random vector $\X \in R^d$ is subGaussian, if there exists $\sigma \in R$ so that:
\begin{equation*}
\E e^{\la \v,  \X - \E \X \ra}  \le e^{\frac{\norm{\v}^2\sigma^2}{2}}, \quad \forall \v \in \R^{d}.
\end{equation*}
\end{definition}
The concentration bounds of subGaussian random vectors/variables depends on the parameter $\sigma$ -- smaller the $\sigma$ better the concentration bounds. While subGaussian distributions arise naturally in several applications, there are settings where the random vectors have nice concentration properties but the subGaussian parameter $\sigma$ is very large (so that applying concentration bounds for general subGaussian random vectors gives loose bounds). In this short note, we consider a related but different class of distributions, called \emph{norm-subGaussian} random vectors and establish tighter concentration bounds for them.

\textbf{Organization}: In Section~\ref{sec:nSG}, we introduce norm subGaussian random vectors and some of their properties and we prove our main results in Section~\ref{sec:results}. We conclude in Section~\ref{sec:conc}.
\section{Norm SubGaussian Random Vector}\label{sec:nSG}
The norm subGaussian random vector is defined as follows.
\begin{definition}\label{def:sGnorm}
A random vector $\X \in \R^d$ is \emph{norm-subGaussian} (or $\nSG(\sigma)$), if $\exists \; \sigma$ so that:
\begin{equation*}
\Pr\left(\norm{\X - \E \X} \ge t \right) \le 2 e^{-\frac{t^2}{2\sigma^2}}, \qquad \forall t \in \R.
\end{equation*}
\end{definition}
Norm subGaussian includes both subGaussian (with a smaller $\sigma$ parameter) and bounded norm random vectors as special cases.
\begin{lemma}\label{lem:examplesGnorm}
There exists absolute constant $c$ so that following random vectors are all $\nSG(c \cdot \sigma)$.
\begin{enumerate}
\item A bounded random vector $\X \in \R^d$ so that $\norm{\X} \le \sigma$.
\item A random vector $\X \in \R^d$, where $\X = \xi \e_1$ and random variable $\xi \in \R$ is $\sigma$-subGaussian.
\item A random vector $\X \in \R^d$ that is $(\sigma/\sqrt{d})$-subGaussian.
\end{enumerate}
\end{lemma}

\begin{proof}
The fact that the first two random vectors are $\nSG(c \cdot\sigma)$ immediately follows from the arguements in scalar version counterparts. For the third random vector, WLOG, assume $\E \X =0$. Let $\{\v_i\}$ be a $1/2$-cover of unit sphere $\S^{d-1}$ (thus $\norm{\v_i} = 1$). By property of subGaussian random vector, we know for each fixed $v_i$:
\begin{equation*}
\Pr(\la\v_i, \X \ra \ge t) \le e^{-\frac{d t^2}{2\sigma^2}}
\end{equation*}
Then let $\v(\X) = \X/\norm{\X }$, since $\{\v_i\}$ is a $1/2$-cover, there always exists a $j(\X)$ so that $\v_{j(\X)}$ in cover and $\norm{\v(\X) - \v_{j(\X)}} \le 1/2$. Therefore, we have:
\begin{align*}
\norm{\X} =& \la \v(\X), \X\ra = \la \v_{j(\X)}, \X\ra + \la \v(\X) - \v_{j(\X)}, \X \ra \\
\le & \la \v_{j(\X)}, \X \ra + \norm{\X}/2
\end{align*}
Rearranging gives $\norm{\X} \le 2\la \v_{j(\X)}, \X \ra$. Finally, the covering number of $1/2$-cover over $\S^{d-1}$ can be upper bounded by $4^d$. Therefore, by union bound:
\begin{align*}
\Pr(\norm{\X } \ge t) \le \Pr(\la \v_{j(\X)}, \X \ra \ge t/2)
\le \Pr(\exists i, \la \v_{i}, \X \ra \ge t/2)
\le 4^d e^{-\frac{d t^2}{8\sigma^2}}
\end{align*}
Now we are ready to check the second claim of Lemma \ref{lem:examplesGnorm}, when 
$t^2 \le 8\sigma^2 \ln 4$, we have,
\begin{equation*}
\Pr(\norm{\X} \ge t)  \le 1
\le 2 e^{-\frac{t^2}{16\sigma^2}}
\end{equation*} 
when $t^2 > 8\sigma^2 \ln 4$, we let $t^2 = 8\sigma^2 \ln 4 + s$ where $s>0$, then:
\begin{equation*}
\Pr(\norm{\X } \ge t) \le 4^d e^{-\frac{dt^2}{8\sigma^2}}
= e^{-\frac{ds}{8\sigma^2}}
\le e^{-\frac{s}{16\sigma^2}}
= 2 e^{-\frac{t^2}{16\sigma^2}}
\end{equation*} 
In sum, this proves that $\X$ is $\nSG(2\sqrt{2}\cdot\sigma)$.
\end{proof}
The following lemma gives equivalent characterizations of norm subGaussian in terms of moments and moment generating function (MGF).
\begin{lemma}[Properties of norm-subGaussian] \label{lem:nSGproperty} For random vector $\X \in \R^d$,
following statements are equivalent up to absolute constant difference in $\sigma$.
\begin{enumerate}
\item Tails: $\Pr(\norm{\X} \ge t) \le 2 e^{-\frac{t^2}{2\sigma^2}}$.
\item Moments: $(\E \norm{\X}^p)^{\frac{1}{p}} \le \sigma \sqrt{p}$ for any $p\in\N$.
\item Super-exponential moment: $\E e^{\frac{\norm{\X}^2}{\sigma^2}} \le e$.
\end{enumerate}
\end{lemma}

\begin{proof}
Note $\norm{\X}$ is a 1-dimensional random variable. This lemma directly follows from the equivalent properties of $1$-dimensional subGaussian, for instance, Lemma 5.5 in \citep{vershynin2010introduction}.
\end{proof}
The following lemma says that if a random vector is $\nSG(\sigma)$, then its norm squared is subexponential and its projection on any direction a is subGaussian random variable.
\begin{lemma}
There is an absolute constant $c$ so that if random vector $\X \in R^d$ is zero-mean $\nSG(\sigma)$, then $\norm{\X}^2$ is $c\cdot\sigma^2$-subExponential, and for any fixed unit vector $\v \in \S^{d-1}$, $\la \v, \X\ra$ is $c\cdot\sigma$-subGaussian.
\end{lemma}


The undesirable thing about the MGF characterization in Lemma~\ref{lem:nSGproperty} is that even if $\X$ is a zero mean random vector, $\norm{\X}$ is not zero mean, so it is difficult to directly work with MGF of $\norm{\X}$. Instead, we first convert the random vector $\X$ to a matrix $\Y$ and characterize the MGF of $\Y$.
\begin{lemma}[MGF Characterization]\label{lem:nSG_MGF}
There is an absolute constant $c$, if random vector $\X \in R^d$ is zero-mean $\nSG(\sigma)$, then let \begin{equation*}
\Y \defeq 
\begin{pmatrix}
0 & \X\trans \\
\X & \bm{0}
\end{pmatrix} \in \R^{(d+1) \times (d+1)}
\end{equation*}
we have $\E e^{\theta \Y} \preceq  e^{c\cdot \theta^2 \sigma^2}\I$ for any $\theta \in \R$.
\end{lemma}

\begin{proof}
Note $\Y$ is a rank-2 matrix whose eigenvalues are $\norm{\X}, -\norm{\X}$, and $\E\Y^{2p+1} = \bm{0}$ for any $p\in \N$.
On the other hand, we also have $\norm{\Y^{2p}} \le \norm{\X}^2p$ for any $p\in\N$. Therefore, by Lemma \ref{lem:nSGproperty}, there exists constant $c$, for any $\theta \in \R$:
\begin{equation*}
\E e^{\theta \Y} = \I + \sum_{p=1}^\infty \frac{\theta^{2p} \E\Y_i^{2p}}{(2p)!}
\preceq \left(1 + \sum_{p=1}^\infty \frac{\theta^{2p} \E\norm{\X}^{2p}}{(2p)!}\right)\I 
\preceq  \left(1 + \sum_{p=1}^\infty \frac{(c\cdot \theta^2 \sigma^2 p)^p }{(2p)!}\right)\I
\preceq e^{c\cdot \theta^2 \sigma^2}\I
\end{equation*}
where in the last inequality we used the fact that $\frac{p^p}{(2p)!} \le \frac{1}{p!}$, this finishes the proof.
\end{proof}

\section{Vector Martingales with SubGaussian Norm}\label{sec:results}
In this section, we will prove our main result (Lemma~\ref{lem:concen_sum_rand}, Corollaries~\ref{cor:concen_sum} and~\ref{cor:concen_sum_randB}) giving concentration bounds for norm subGaussian random vectors. The main tool we use is Lieb's concavity theorem.
\begin{theorem}[\cite{tropp2012user}] \label{thm:Lieb}
Let $\A$ be a fixed symmetric matrix, and let $\Y$ be a random symmetric matrix. Then,
\begin{equation*}
\E \tr(\exp(\A + \Y)) \le \tr \exp(\A + \log (\E e^{\Y}))
\end{equation*}
\end{theorem}
We will prove our concentration result for norm subGaussian random vectors in a general setting where the subGaussian parameter $\sigma_i$ for the $i^{\textrm{th}}$ vector can itself be a random variable.
\begin{condition} \label{cond:subGmartingale}
Let random vectors $\X_1, \ldots, \X_n \in \R^d$, and corresponding filtrations $\F_i = \sigma(\X_1, \ldots, \X_i)$ for $i\in [n]$ satisfy that $\X_i |\F_{i-1}$ is zero-mean $\nSG(\sigma_i)$ with $\sigma_i \in \F_{i-1}$. i.e.,
\begin{equation*}
\E [\X_i |\F_{i-1}]  = 0, \quad \Pr\left(\norm{\X_i} \ge t | \F_{i-1}\right) \le 2 e^{-\frac{t^2}{2\sigma_i^2}}, \qquad \forall t \in \R, \forall i \in [n].
\end{equation*}
\end{condition}

\begin{lemma} \label{lem:concen_sum_rand}
There exists an absolute constant $c$ such that if $\X_1, \ldots, \X_n \in \R^d$ satisfy condition \ref{cond:subGmartingale}, then for any fixed $\delta >0$, $\theta>0$, with probability at least $1-\delta$:
\begin{equation*}
\norm{\sum_{i=1}^n \X_i}  \le  c\cdot \theta \sum_{i=1}^n \sigma_i^2 + \frac{1}{\theta} \log \frac{2d}{\delta}
\end{equation*}
\end{lemma}

\begin{proof}
According to Lemma \ref{lem:nSG_MGF}, there exists an absolute constant $c$ so that $
\E[e^{\theta \Y_i}|\F_{i-1}] \preceq  e^{c\cdot\theta^2 \sigma_i^2}\I$ holds for any $i\in [n]$.
Therefore, we have:
\begin{align*}
&\E \tr \exp(-c\cdot \theta^2 \sum_{i=1}^n \sigma_i^2 \I + \theta \sum_{i=1}^n \Y_i)
= \E\{ \E[\tr \exp(-c\cdot \theta^2 \sum_{i=1}^n \sigma_i^2 \I  + \theta \sum_{i=1}^n \Y_i)|\F_{n-1}] \} \\
\overset{(1)}\le& \E \tr \exp(-c\cdot \theta^2 \sum_{i=1}^n \sigma_i^2 \I  + \theta \sum_{i=1}^{n-1} \Y_i + \log \E[e^{\theta \Y_n} |\F_{n-1}]) 
\overset{(2)}{\le} \E \tr \exp(-c\cdot \theta^2 \sum_{i=1}^{n-1} \sigma_i^2 \I  + \theta \sum_{i=1}^{n-1} \Y_i) \\
\le& \ldots \le \tr \exp(0 \I) = d 
\end{align*}
where step (1) is due to Theorem~\ref{thm:Lieb}, and step (2) used the fact that if matrix $\A \preceq \B$, then $e^{\mat{C} + \A} \preceq e^{\mat{C} + \B}$. On the other hand, since identity matrix commutes with any matrix, we know:
$$\exp(-c\cdot \theta^2 \sum_{i=1}^n \sigma_i^2 \I + \theta \sum_{i=1}^n \Y_i) = \exp(-c\cdot\theta^2 \sum_{i=1}^n \sigma_i^2) \cdot \exp(\theta \sum_{i=1}^n \Y_i) $$
Therefore, for any $t \ge 0$, $\theta \ge 0$, by Markov's inequality, we have:
\begin{align*}
&\Pr\left[\norm{\sum_{i=1}^n \X_i} \ge c\cdot\theta \sum_{i=1}^n \sigma_i^2 + t/\theta\right]
\overset{(1)}{=} \Pr\left[\norm{\sum_{i=1}^n \Y_i} \ge c\cdot\theta \sum_{i=1}^n \sigma_i^2 + t/\theta\right]\\
\overset{(2)}{=}& 2\Pr\left[\lambda_{\max}\left(\sum_{i=1}^n \Y_i\right) \ge c\cdot\theta \sum_{i=1}^n \sigma_i^2 + t/\theta\right] 
= 2\Pr\left[\lambda_{\max}\left(e^{\theta \sum_{i=1}^n \Y_i}\right) \ge e^{c\cdot\theta^2 \sum_{i=1}^n \sigma_i^2 + t}\right] \\
\le& 2\Pr\left[\tr\left(e^{\theta \sum_{i=1}^n \Y_i}\right) \ge e^{c\cdot\theta^2 \sum_{i=1}^n \sigma_i^2 + t}\right] 
\le 2 e^{-t} \E \tr\left(e^{-c\cdot \theta^2 \sum_{i=1}^n \sigma_i^2 \I  + \theta \sum_{i=1}^n \Y_i}\right)
\le 2d e^{-t}
\end{align*}
where step (1) is because $\sum_{i=1}^n \Y_i$ is a rank-2 matrix whose eigenvalues are $\norm{\sum_{i=1}^n \X_i}, -\norm{\sum_{i=1}^n \X_i}$;
step (2) is due to all preconditions are symmetric with respect to 0. Finally, setting RHS equal to $\delta$, we finish the proof.
\end{proof}

\begin{corollary} [Hoeffding type inequality for norm-subGaussian] \label{cor:concen_sum}
There exists an absolute constant $c$ such that if $\X_1, \ldots, \X_n \in \R^d$ satisfy condition \ref{cond:subGmartingale} with fixed $\{\sigma_i\}$, then for any $\delta >0$, with probability at least $1-\delta$:
\begin{equation*}
\norm{\sum_{i=1}^n \X_i}  \le c\cdot \sqrt{\sum_{i=1}^n \sigma_i^2 \log \frac{2d}{\delta}}
\end{equation*}
\end{corollary}

\begin{proof}
Since now $\{\sigma_i\}$ are fixed which are not random, we can pick $\theta$ in Lemma \ref{lem:concen_sum_rand} as a function of $\{\sigma_i\}$. Indeed, pick 
$\theta = \sqrt{\frac{1}{\sum_{i=1}^n \sigma_i^2} \log \frac{2d}{\delta}}$ finishes the proof.
\end{proof}

\begin{corollary} \label{cor:concen_sum_randB}
There exists an absolute constant $c$ such that if $\X_1, \ldots, \X_n \in \R^d$ satisfy condition \ref{cond:subGmartingale}, then for any fixed $\delta >0$, and $B > b > 0$, with probability at least $1-\delta$:
\begin{equation*}
\text{either }\sum_{i=1}^n \sigma_i^2 \ge B \quad \text{or} \quad
\norm{\sum_{i=1}^n \X_i}  \le  c\cdot \sqrt{\max\{\sum_{i=1}^n \sigma_i^2, b\}\cdot (\log \frac{2d}{\delta} + \log\log \frac{B}{b})}
\end{equation*}
\end{corollary}

\begin{proof} For simplicity, denote log factor $\iota \defeq \log \frac{2d}{\delta} + \log\log \frac{B}{b}$ By Lemma \ref{lem:concen_sum_rand}, we know for any fixed $\theta$, with probability $1-\delta \cdot \log^{-1}(B/b)$, we have:
\begin{equation*}
\norm{\sum_{i=1}^n \X_i}  \le  c\cdot \theta \sum_{i=1}^n \sigma_i^2 + \frac{\iota}{\theta}
\end{equation*}
Construct two sets of $\Psi  = \{\psi_1, \ldots, \psi_s\}$ and $\Theta = \{\theta_1, \ldots, \theta_s\}$, where
$\psi_j = 2^{j-1} \cdot b$ and $\theta_j = \sqrt{\frac{\iota}{\psi_j}}$
with last element $\psi_s \le B$, $2\psi_s > B$. It is easy to see $|\Psi| = |\Theta| \le \log (B/b)$. By union bound, we have with probability $1-\delta$: 
\begin{equation*}
\norm{\sum_{i=1}^n \X_i}  \le  \min_{j\in [s]} \left[c\cdot \theta_j \sum_{i=1}^n \sigma_i^2 + \frac{\iota}{\theta_j}\right]
\end{equation*}

Consider following two cases: (1) $\sum_{i=1}^n \sigma_i^2 \in [b, B]$. Then, there exists $j\in[s]$ such that $\psi_j \le \sum_{i=1}^n \sigma_i^2 < 2\psi_j$:
\begin{equation*}
\norm{\sum_{i=1}^n \X_i}  \le  c\cdot\theta_j \sum_{i=1}^n \sigma_i^2 + \frac{\iota}{\theta_j}
= c\cdot \sqrt{\frac{\iota}{\psi_j}} \sum_{i=1}^n \sigma_i^2 + \iota \cdot\sqrt{\frac{\psi_j}{\iota}}
\le (2c + 1)\sqrt{ \sum_{i=1}^n \sigma_i^2 \cdot \iota}
\end{equation*}
(2) $\sum_{i=1}^n \sigma_i^2 \in [0, b)$. In this case we know $\psi_1 =b$ and:
\begin{equation*}
\norm{\sum_{i=1}^n \X_i}  \le  c\cdot\theta_1 \sum_{i=1}^n \sigma_i^2 + \frac{\iota}{\theta_1}
= c\cdot\sqrt{\frac{\iota}{b}} \sum_{i=1}^n \sigma_i^2 + \iota \cdot\sqrt{\frac{b}{\iota}}
\le (2c + 1)\sqrt{b \cdot \iota}
\end{equation*}

Combining two cases we finish the proof.

\end{proof}

\section{Conclusion}\label{sec:conc}
In this short note, we introduced the notion of norm subGaussian random vectors, which include subGaussian random vectors and bounded random vectors as special cases. While it is true that $\textrm{subGaussian}\left(\frac{\sigma}{\sqrt{d}}\right) \subseteq \nSG(\sigma) \subseteq \textrm{subGaussian}(\sigma)$, applying concentration bounds for $\textrm{subGaussian}(\sigma)$ would yield bounds which have at least linear dependence on $d$. In contrast, the bounds we develop (in Lemma~\ref{lem:concen_sum_rand} and Corollaries~\ref{cor:concen_sum} and~\ref{cor:concen_sum_randB}) have only logarithmic dependence on $d$. It is not clear if this logarithmic dependence is tight -- totally eliminating this dependence is an interesting open problem.


\section*{Acknowledgements}
We thank Gabor Lugosi and Nilesh Tripuraneni for helpful discussions.

\bibliographystyle{plainnat}
\bibliography{saddle}

\end{document}